\documentclass[12pt]{amsart} \usepackage{amsmath,centernot,bm}
\usepackage{amssymb} \usepackage{amsthm}
\usepackage[hidelinks]{hyperref}

\newcommand{\tpdf}{\texorpdfstring}

\newcommand{\Bohr}{\operatorname{Bohr}}

\newcommand{\mb}{\mathbf}
\newcommand{\supp}{\operatorname{supp}}
\newtheorem{theorem}{Theorem}
\newtheorem{lemma}[theorem]{Lemma}
\newtheorem{proposition}[theorem]{Proposition}
\newtheorem{corollary}[theorem]{Corollary}

\numberwithin{theorem}{section}
\numberwithin{equation}{section}

\theoremstyle{definition}

\newtheorem{observation}[theorem]{Observation}

\newtheorem{remark}[theorem]{Remark}

\title[Bohr nhoods in generalized difference sets]{Bohr neighborhoods in generalized difference sets}

\author{John T. Griesmer}
\email{jtgriesmer@gmail.com}
\usepackage{amstext}

\begin{document}
	
\begin{abstract}
  If $A$ is a set of integers having positive upper Banach density and $r,s,t$ are nonzero integers whose sum is zero, a theorem of Bergelson and Ruzsa says that the set $rA+sA+tA:=\{ra_1+sa_2+ta_3:a_i\in A\}$ contains a Bohr neighborhood of zero.  We prove the natural generalization of this result for subsets of countable abelian groups and more summands.
\end{abstract}

	\maketitle

\section{Bohr neighborhoods in iterated difference sets}\label{sec:Intro}

Let $\mathcal S^1:=\{z\in \mathbb C: |z|=1\}$ be the group of complex numbers with modulus $1$, with the usual topology and the operation of multiplication.

Let $G$ be a topological abelian group.  A \emph{character} of $G$ is a continuous homomorphism $\chi:G\to \mathcal S^1$.  Of course, when $G$ is discrete, every such homomorphism is continuous.  A \emph{trigonometric polynomial} is linear combination of characters, i.e. a function of the form $\sum_{j=1}^d c_j\chi_j$, where $\chi_j\in \widehat{G}$.  A function $\psi$ on $G$ is \emph{uniformly almost periodic} if it is a uniform limit of trigonometric polynomials.

The set $\widehat{G}$ of characters of $G$ form an abelian group under  pointwise multiplication. When $G$ is discrete and $\widehat{G}$ is given the topology of pointwise convergence, $\widehat{G}$ is a compact abelian group.  When $G$ is compact and $\widehat{G}$ is given the topology of uniform convergence, $\widehat{G}$ is discrete.

Let $d\in \mathbb N$ and $\chi_1,\dots,\chi_d\in \widehat{G}$. The \emph{Bohr neighborhood of $0$} in $G$ having rank $d$ and radius $\varepsilon$ determined by $\{\chi_1,\dots,\chi_d\}$ is
\[
\Bohr(\chi_1,\dots,\chi_d;\varepsilon):=\{g: |\chi_i(g)-1|<\varepsilon, 1\leq i \leq d\}.
\]
\begin{remark}\label{rem:APtoBohr} If $\psi$ is a real valued trigonometric polynomial with $\psi(0)>0$, it is easy to verify that $\supp(\phi):=\{g\in G:\psi(g)\neq 0\}$ contains a Bohr neighborhood of $0$.  Since almost periodic functions may be uniformly approximated by trigonometric polynomials, the same holds when $\psi$ is uniformly almost periodic.
\end{remark}

A classical theorem of Bogoliouboff \cite{Bogoliouboff} states that if $G$ is a countable abelian group and $G = C_1\cup C_2\cup \dots \cup C_r$ is a finite cover of $G$, then for at least one $i$, the iterated difference set $(C_i-C_i)-(C_i-C_i):=\{(x-y)-(z-w): x,y,z,w\in C_i\}$ contains a Bohr neighborhood of $0$ with radius and rank depending only on $r$.  Note that Bogoliouboff's theorem implies that if finitely many translates of $C$ cover $G$, then $(C-C)-(C-C)$ contains a Bohr neighborhoof of $0$.   F{\o}lner \cite{Folner} generalizes this to say that $(A-A)-(A-A)$ contains a Bohr neighborhood of $0$ under the weaker hypothesis that $d^*(A)>0$, where $d^*$ denotes upper Banach density (see \S\ref{sec:Densities} for the definition).  See Theorem 5.7.2 of Part II of \cite{RuzsaBook} for exposition of Bogoliouboff's and F{\o}lner's results.

Bergelson and Ruzsa extended F{\o}lner's theorem as follows.
\begin{theorem}[\cite{BergelsonRuzsa}, Theorem 6.1]\label{thm:BergelsonRuzsa}
Let $r, s, t$ be nonzero integers with $r+s+t=0$, and let $A\subset \mathbb Z$ have $d^*(A)>0$.  Then $rA+sA+tA:=\{ra_1+sa_2+ta_3: a_i\in A\}$ contains a Bohr neighborhood of $0$, with rank and radius depending only on $r, s, t$, and $d^*(A)$.
\end{theorem}

Our main result is Theorem \ref{thm:GeneralBogoliouboff}, a partial generalization of Theorem \ref{thm:BergelsonRuzsa} to arbitrary countable abelian groups and more summands.  We say ``partial generalization'' because we do not bound the rank and radius of the Bohr neighborhood.

Let $G$ be a countable (discrete) abelian group.  Let $q$ be the exponent of $G$, meaning $q$ is the least $m\in \mathbb N$ such that $mg=0$ for every $g\in G$.  If there is no such $m$, we say $G$ has exponent $0$.

If $\vec{c}=(c_1,\dots,c_d)\in \mathbb Z^d$ and $A\subset G$, we define
\[\vec{c}\cdot A:=\Bigl\{\sum_{i=1}^d c_ia_i: a_1,\dots,a_d\in A \text{ \textup{are mutually distinct}}\Bigr\}\cup\{0\}\]
\begin{theorem}\label{thm:GeneralBogoliouboff}
  Let $d\in \mathbb N$, $d\geq 3$, and let $c_1,\dots, c_d\in \mathbb Z$ be such that
  \begin{enumerate}
  \item $c_iG$ has finite index in $G$ for each $i$,
  \item\label{hyp:cicj}  $(c_{d-1}+c_{d})G$, has finite index in $G$,  and
   \item $q$ divides $\sum c_i$.
\end{enumerate}
If $A\subset G$ has $d^*(A)>0$, then $\vec{c}\cdot A$ contains a Bohr neighborhood of $0$ in $G$.
\end{theorem}

Letting $p$ be an odd prime, $G=\mathbb F_p^\omega:=$ the direct sum of countably many copies of $\mathbb Z/p\mathbb Z$, and $c_1=\dots =c_d=1$, we get the following corollary as a special case.

\begin{corollary}
  If $p$ is an odd prime, $p|d\in \mathbb N$, and $A\subset \mathbb F_p^\omega$ has $d^*(A)>0$, then
  \[\{a_1+\cdots +a_d: a_i\in A \text{ are mutually distinct}\}\] contains a Bohr neighborhood of $0$.
\end{corollary}

\begin{remark}
  The rank and radius of $B$ in the conclusion of Theorem \ref{thm:GeneralBogoliouboff} can be bounded by functions of the indices of $c_iG$ in $G$ and $d^*(A)$, but we will not prove such bounds in this article.

  While Theorem \ref{thm:BergelsonRuzsa} does not insist that the sums be formed with mutually distinct $a_i$, the applications motivating Theorem \ref{thm:GeneralBogoliouboff} make it convenient to do so.

  The hypothesis (\ref{hyp:cicj}) can be generalized to ``there exists $i<j\leq d$ such that $(c_i+c_j)G$ has finite index in $G$,'' since the $c_i$ can be reordered without changing the value of $\vec{c}\cdot A$.
\end{remark}

\begin{remark}  We are currently unable to decide whether the hypothesis (\ref{hyp:cicj}) can be omitted in Theorem \ref{thm:GeneralBogoliouboff}.  The other hypotheses cannot be omitted, as the following three examples demonstrate.

   \begin{enumerate} \item[(i)]  The hypothesis that $d\geq 3$ cannot be weakened to $d\geq 2$: setting $d=2$ and $c_1=1$, $c_2=-1$, the set $\vec{c}\cdot A$ is simply $A-A$.  An example of a set $A\subset \mathbb Z$ having $d^*(A)>0$ such that $A-A$ does not contain a Bohr neighborhood of $0$ is provided by the main result of \cite{Kriz} - see \cite{HegyvariRuzsa} for an explanation of why \cite{Kriz} implies this.  Such an example in $\mathbb F_2^\omega$ is constructed in \cite{ForrestThesis}, and in $\mathbb F_p^\omega$ for odd primes $p$ in \cite{griesmer2017bohr}.  In fact there is an $A\subset \mathbb Z$ with $d^*(A)>0$ where $A-A$ does not contain a Bohr neighborhood of \emph{any} $n\in \mathbb Z$, as proved in \cite{griesmer2020separating}.

   \item[(ii)] The hypothesis that $c_iG$ has finite index in $G$ cannot be omitted. To see this, we use the fact that for each prime $p$ and $\varepsilon>0$, the group $G=\mathbb F_p^{\omega}$ contains a set $A$ having $d^*(A)>\frac{1}{2}-\frac{1}{2p}-\varepsilon$ such that $A-A$ does not contain a Bohr neighborhood of \emph{any} element of $G$ (see \cite{griesmer2017bohr}).  Setting $c_1=p$, $c_2=1$ and $c_3=-1$, we have $c_1+c_2+c_3=p$, but $c\cdot A \subseteq pA+A-A =A-A$.


  \item[(iii)]  The hypothesis that $q|\sum c_i$ cannot be omitted.  For example, in $\mathbb Z$, the set $A$ of odd integers has $d^*(A)=1/2$, and $A-A+A=A$, which does not contain a Bohr neighborhood of $0$. Nevertheless, $A-A+A$ contains Bohr neighborhoods of many elements of $A$ whenever $d^*(A)>0$; see \cite{HegyvariRuzsa}, \cite{BjorklundGriesmer}.
%
\end{enumerate}

\end{remark}

The next lemma deals with a technicality that arises from our approach.

\begin{lemma}\label{lem:IndexBohr}
  If $B$ is a Bohr neighborhood of $0$ in an abelian group $G$ and $c\in \mathbb Z$ is such that $cG$ has finite index in $G$, then $cB$ is also a Bohr neighborhood of $0$ in $G$.
  \end{lemma}
  \begin{proof}
  This is somewhat cumbersome to verify directly, but easy to see as a consequence of Bogoliouboff's theorem. If $B$ is a Bohr neighborhood of $0$ in $G$, then $B$ contains a set of the form $(A-A)-(A-A)$, where $A$ is a Bohr neighborhood of $0$ in $G$ having radius smaller than the radius of $B$.  In particular, finitely many translates of $A$ cover $G$.  Since $cG$ has finite index in $G$, we have that finitely many translates of $cA$ cover $G$. Now $cB$ contains $(cA-cA)-(cA-cA)$, which must contain a Bohr neighborhood of $0$, by Bogoliouboff's theorem.
\end{proof}

\subsection*{Notation for integrals}  We adopt the usual notation for integrals: $\int f\, d\mu$ is the integral of $f$ with respect to $\mu$, or $\int f(x)\, d\mu(x)$, if we need to specify a variable. When $\mu$ is Haar measure on a compact abelian group $K$, we may write $\int f(x) \, dx$ in place of $\int f(x)\, d\mu(x)$, to shorten the appearance of iterated integrals.

We will prove Theorem \ref{thm:GeneralBogoliouboff} by proving that there is an almost periodic function $\phi$ with $\phi(0)>0$ and $c_d\supp(\phi)\subset \vec{c}\cdot A$ (cf. Remark \ref{rem:APtoBohr} and Lemma \ref{lem:IndexBohr}).  The following lemma specifies the form of $\phi$.

\begin{lemma}\label{lem:Reduction}
  Let $A\subset G$ and let $c_1,\dots, c_d$ satisfy the hypothesis of Theorem \ref{thm:GeneralBogoliouboff}. Then there is a compact abelian group $K$ with Haar probability measure $m$, an $m$-measurable  $\tilde{f}:K\to [0,1]$ with $\int \tilde{f}\, dm=d^*(A)$, and a homomorphism $\rho:G\to K$ with $\overline{\rho(G)}=K$ such that
  \[
  \phi(t):=\int_{K^d} \tilde{f}(k+c_ds_1)\tilde{f}(k+c_ds_2)\cdots \tilde{f}(k+c_{d}s_{d-1}) \tilde{f}\Bigl(k+\rho(t)-\sum_{i=1}^{d-1} c_is_i\Bigr) \, dk\, ds_1\, ds_2\, \dots \, ds_{d-1}
  \]
  satisfies $c_d\supp(\phi)\subset \vec{c}\cdot A$.
\end{lemma}

\begin{proof}[Proof of Theorem \ref{thm:GeneralBogoliouboff}]

Assuming Lemma \ref{lem:Reduction}, it suffices to prove that the $\phi$ defined therein is almost periodic and satisfies $\phi(0)>0$.  The almost periodicity of $\phi$ follows by approximating (in $L^2(\mu)$) the function $\tilde{f}$  in the definition by a trigonometric polynomial $p:K\to [0,1]$, so that $\|\tilde{f}-p\|<\varepsilon/{2^d}$ and observing that the resulting function of $t$ is a trigonometric polynomial on $G$ differing (uniformly) from $\phi$ by at most $\varepsilon.$

To prove that $\phi(0)>0$, note that the function on $K^{d-1}$ defined by
\[
(s_1,\dots, s_{d-1}) \mapsto \int \tilde{f}(k+s_1) \cdots \tilde{f}(k+s_{d-1})\tilde{f}\Bigl(k-\sum_{i=1}^{d-1} c_is_i\Bigr)\, dk
\]
is continuous, and is equal to $\int \tilde{f}(x)^d\, dx$ when $s_1=s_2=\cdots =s_d=0$.  The integrand in the definition of $\phi$ is therefore positive for a neighborhood of values $(s_1,\dots,s_{d-1})$ near $0$.  Consequently, the integral defining $\phi$ does not vanish at $t=0$.  \end{proof}

\begin{remark}
  The argument in the second paragraph of the above proof is a key ingredient in ergodic theoretic proofs of Roth's theorem on three-term arithmetic progressions.  In fact, Roth's theorem is a corollary of our proofs, although we do not discuss this implication here. The proof of Theorem 6.1 in \cite{BergelsonRuzsa} uses Roth's theorem in a key step.
\end{remark}

\section{Sumsets as supports of convolutions}

In this section we state Theorem \ref{thm:DynamicalVersion}, an ergodic theoretic version of Lemma \ref{lem:Reduction}, and use it to prove the lemma.  Theorem \ref{thm:DynamicalVersion} will be proved in \S\ref{sec:ProofOfErgodicVersion}, after some preliminary machinery is developed.

Our first step in the proof of Lemma \ref{lem:Reduction} is to find a kind of convolution supported on the sumset in question.  For the sake of explanation, we first consider two summands.

Given $A, B\subset G$, consider their characteristic functions $1_A$ and $1_B$.  Note that a fixed $g$ lies in $A+B$ if and only if there is an $h\in G$ such that $1_A(h)1_B(g-h)>0$. Often the existence of such an $h$ is proved by averaging, i.e. proving that there is a F{\o}lner sequence $F_n$ with $\lim_{N\to \infty} \frac{1}{|F_n|}\sum_{h\in F_n} 1_A(h)1_B(g-h)>0$.  This limit is a kind of convolution of $1_A$ and $1_B$.

We want to find a convolution supported on the set $\vec{c}\cdot A$ and write it in a form suited to the hypotheses of Theorem \ref{thm:GeneralBogoliouboff}.  The next observation does so in a somewhat cumbersome way.

\begin{lemma}\label{lem:Convolution} Let $\vec{c}\in \mathbb Z^d$ satisfy hypothesis (3) of Theorem \ref{thm:GeneralBogoliouboff}, let $A\subset G$,  and write $1_A$ for its characteristic function.  Define $J: G^{d+1}\to [0,1]$  and $J_0\subset G$ by
\begin{align*}
&J(h,g_1,\dots,g_{d-1};t):=1_A(h+c_dg_1)1_A(h+c_dg_2)\cdots 1_A(h+c_{d}g_{d-1}) 1_A\Bigl(h+t-\sum_{i=1}^{d-1} c_ig_i\Bigr)\\
&J_0:=\{t\in G: \exists h, g_1,\dots, g_{d-1} \in G \text{ with } J(h,g_1,\dots, g_{d-1};t)>0,  c_ig_i \text{  mutually distinct}\},
\end{align*}
 we have
\begin{equation}\label{eqn:J0A}
c_dJ_0 \subset \vec{c}\cdot A.
\end{equation}
Furthermore, setting
\begin{align*}
&\tilde{J}(g_1,\dots,g_{d-1};t):=d^*(\{h:J(h,g_1,\dots,g_{t-1};t)>0\})\\
&\tilde{J}_0:=\{t\in G: \exists g_1,\dots, g_{d-1}\in G \text{ with } \tilde{J}(g_1,\dots,g_{d-1};t)>0, c_ig_i \text{ mutually distinct}\},
\end{align*}
we have $c_d\tilde{J}_0\subset \vec{c}\cdot A$.
\end{lemma}

\begin{proof}
We first prove the containment (\ref{eqn:J0A}).     Suppose $h$, $g_i$, and $t\in G$ are such that $J(h,g_1,\dots,g_{d-1};t)>0$.  Then each of $a_i:=h+c_dg_i$, and $a_d = h+ t - \sum_{i=1}^{d-1}c_ig_i$ lie in $A$, and
\begin{align*}
c_1a_1+\cdots + c_da_d &= (c_1+\cdots + c_d)h + c_d\Bigl(\sum_{i=1}^{d-1} c_ig_i\Bigr) + c_d\Bigl(t-\sum_{i=1}^{d-1}c_ig_i\Bigr)\\
&= c_dt,
\end{align*}
so $c_dt\in \vec{c}\cdot A$.

The second assertion of the lemma follows from the observation that if $\tilde{J}(g_1,\dots,g_{d-1};t)>0$, then there are infinitely many $h$ with $J(h,g_1,\dots,g_{d-1};t)>0$.
\end{proof}

Theorem \ref{lem:Reduction} will be deduced from the following ergodic theoretic version; see \S\ref{sec:Ergodic} for definitions.  These theorems are connected by a Proposition \ref{prop:Correspondence}, a standard correspondence principle.

\begin{theorem}\label{thm:DynamicalVersion}
Let $(X,\mu,T)$ be an ergodic measure preserving $G$-system and let $(c_1,\dots,c_d)\in \mathbb Z^d$ be coefficients satisfying hypotheses (1)-(3) of Theorem \ref{thm:GeneralBogoliouboff}.  If $f_1,\dots, f_d:X\to [0,1]$ are measurable functions, $L:G^d\to [0,1]$ is defined as
\begin{equation}\label{eqn:Ldef}
L(g_1,\dots,g_{d-1};t):= \int T_{t-\sum_{j<d} c_jg_j}f_{d}\cdot \prod_{j< d} T_{c_dg_j}f_{j} \, d\mu
\end{equation}
 and $I(t)$ is defined as the iterated limit
\begin{equation}\label{eqn:defIg}
I(t):=\lim_{n_1\to \infty} \dots  \lim_{n_{d-1}\to \infty} \frac{1}{|F_{n_1}|\cdots |F_{n_{d-1}}|} \sum_{\substack{g_1\in F_{n_1}\\ \dots \\ g_{d-1}\in F_{n_{d-1}}}}L(g_1,\dots,g_{d-1};t)
\end{equation}
 then there are \begin{enumerate}
\item[$\bullet$] a compact metrizable abelian group $K$ with Haar measure $m$,

\item[$\bullet$] a homomorphism $\rho:G\to K$ having $\overline{\rho(G)}=K$,

\item[$\bullet$] functions $\tilde{f}_i:K\to [0,1]$ with $\int \tilde{f}_i\, dm = \int f_i\, d\mu $ for $i\leq d$
\end{enumerate}
such that
\begin{equation}\label{eqn:IgReduction}
I(t):=\int \tilde{f}_d\Bigl(k+\rho(t)-\sum_{j<d}c_js_j\Bigr) \prod_{j<d} \tilde{f}_j(k+c_ds_j) \, dk\, ds_1\, \dots\, ds_{d-1}.
\end{equation}
Furthermore, the $\tilde{f}_i$ depend only on the $f_i$, so if $f_1=\cdots=f_d$, then $\tilde{f}_1=\cdots=\tilde{f}_d$.
\end{theorem}

We prove Theorem \ref{thm:DynamicalVersion} in \S\ref{sec:ProofOfErgodicVersion}.

\begin{remark}\label{rem:PositiveLimitImpliesPositive}
  With $L$ and $I$ as in Theorem \ref{thm:DynamicalVersion}, let
   \[
   L_0:=\{t\in G: \exists g_1,\dots, g_{d-1} \text{ with } L(g_1,\dots,g_{d-1};t)>0 \text{ and } c_ig_i \text{ mutually distinct}\}
   \]
   Then $\{t\in G: I(t)>0\}\subset L_0$.
\end{remark}

%

\subsection{Correspondence principle}

The connection between Theorem \ref{thm:DynamicalVersion} and Lemma \ref{lem:Reduction} is made through the following variation of Furstenberg's correspondence principle \cite{F77}.  Our proof is simply the construction done in \cite[Theorem 2.1]{BergelsonMcCutcheon}, together with a fragment of the proof of \cite[Lemma 5.1]{BjorklundProductSet} which shows that choosing an \emph{extreme} invariant mean in that construction results in an ergodic system.
\begin{proposition}\label{prop:Correspondence}
  Let $A\subset G$ have $d^*(A)=\delta$.  There is an ergodic $G$-system $(X,\mu,T)$ and a $\mu$-measurable $f:X\to [0,1]$ with $\int f\, d\mu=\delta$ such that for all $g_1,g_2,\dots,g_d\in G$, we have
  \begin{equation}\label{eqn:Correspondence}
    d^*\bigl(\bigl\{h:1_A(h + g_1)1_A(h+g_2)\dots 1_A(h+g_d)>0\bigr\}\bigr) \geq \int \prod_{i=1}^d f(T_{g_i}x)\, d\mu(x)
  \end{equation}
\end{proposition}

Setting $g_d=t-\sum_{j<d}c_jg_j$ in Proposition \ref{prop:Correspondence} immediately yields the following corollary.

\begin{corollary}\label{cor:Correspondence2}
  If $A\subset G$, $c_i\in \mathbb Z$, and $\tilde{J}$ is defined as in Lemma \ref{lem:Convolution}, then there is an ergodic $G$-system $(X,\mu,T)$ and $f:X\to [0,1]$ such that
\begin{equation}\label{eqn:JgeqL}
\tilde{J}(g_1,\dots,g_{d-1};t) \geq L(g_1,\dots,g_{d-1};t) \qquad \text{for all } g_i, t\in G,
\end{equation}
where $L$ is as defined in (\ref{eqn:Ldef}), with $f_1=\dots=f_d=f$.
\end{corollary}

We postpone the proof of Proposition \ref{prop:Correspondence} and now prove Lemma \ref{lem:Reduction}

\begin{proof}[Proof of Lemma \ref{lem:Reduction}]
  Let $G$ be a countable abelian group, let $c_1,\dots, c_d\in \mathbb Z$ satisfy hypotheses (1)-(3) in Theorem \ref{thm:GeneralBogoliouboff}, and let $A\subset G$ have $d^*(A)>0$.  Let $\tilde{J}$ and $\tilde{J}_0$ be as in Lemma \ref{lem:Convolution}, so that $c_d \tilde{J}_0\subset \vec{c}\cdot A$, by Lemma \ref{lem:Convolution}.  Apply Corollary \ref{cor:Correspondence2} to find an ergodic $G$-system $(X,\mu,T)$ and $f:X\to [0,1]$ with $\int f\, d\mu = d^*(A)$ satisfying inequality (\ref{eqn:JgeqL}).  This inequality implies $L_0\subset \tilde{J}_0$. By Theorem \ref{thm:DynamicalVersion} and Remark \ref{rem:PositiveLimitImpliesPositive}, we have that $\{t\in G: I(t)>0\}\subset L_0$.  Combining the preceding containments, we have
  \[
  c_d\{t\in G: I(t)>0\} \subset \vec{c}\cdot A.
  \]
  The form of $I(t)$ stated in the conclusion of Theorem \ref{thm:DynamicalVersion} means we can take $\phi(t):=I(t)$ to satisfy the conclusion of Lemma \ref{lem:Reduction}. \end{proof}

Before proving Proposition \ref{prop:Correspondence}, we summarize some background on invariant means; readers can consult \cite{BjorklundProductSet} and \cite{BergelsonMcCutcheon} for further details.  

An \emph{invariant mean} on $\ell^\infty(G)$ is a positive linear functional $\eta$ satisfying $\eta(1_G)=1$; and for every $\phi\in \ell^\infty(G)$ and $t\in G$, we have $\eta(\phi_t)=\eta(\phi)$, where $\phi_t$ is the translate of $\phi$ by $t$ (i.e. $\phi_t(g)=\phi(g+t)$). It is easy to prove that for every F{\o}lner sequence $(F_n)_{n\in \mathbb N}$, there is an invariant mean  $\eta$ such that $\lim_{n\to\infty} \frac{1}{|\Phi_n|}\sum_{g\in \Phi_n} \phi(g) = \eta(\phi)$ for every bounded $\phi:G\to \mathbb C$ for which the limit exists.  This yields the following characterization of upper Banach density: for all $A\subset G$, we have
\[
d^*(A)=\max\{\eta(1_A): \eta \text{ is an invariant mean on } \ell^\infty(G) \}
\]
Since the set of invariant means is a compact convex set in the weak$^*$ topology on the dual of $\ell^\infty(G)$, it contains extreme points.  In other words, there are invariant means which cannot be written as a nontrivial convex combination of other invariant means.  Furthermore, every invariant mean can be written as a convex combination of extreme invariant means, and it follows that for every $A\subset G$, there is an extreme invariant mean $m$ with $m(1_A)=d^*(A)$.

The key identity for ergodicity is most easily expressed using the following notation: given an invariant mean $\eta$ and $\phi\in \ell^\infty(G)$, we write $\int \phi(g)\, d\eta(g)$ for $\eta(\phi)$ (even though the ``integral'' is only finitely additive).  This allows us to write expressions such as $\int \int \psi(g)\phi(g+h)\, d\eta(g)\, d\kappa(h)$ in place of the more cumbersome ``$\kappa(\Phi)$, where $\Phi(h):= \eta(\psi\cdot \phi_h)$.''  The key identity is stated as Scholium 4.4 of \cite{BjorklundProductSet}: if $m$ is an extreme invariant mean and $\eta$ is any invariant mean then for all $\phi_1,\phi_2\in \ell^\infty(G)$, we have
\begin{equation}\label{eqn:ErgodicMean}\int \phi_1(g)\phi_2(g+h) \, dm(g)\, d\eta(h) = m(\phi_1)m(\phi_2).
\end{equation}

\begin{proof}[Proof of Proposition \ref{prop:Correspondence}]
  The proof of \cite[Theorem 2.1]{BergelsonMcCutcheon} produces a (not necessarily ergodic) $G$-system satisfying (\ref{eqn:Correspondence}).  We will repeat the proof therein with one modification: begin by choosing an \emph{extreme} invariant mean on $\ell^\infty(G)$ with $m(1_A)=d^*(A)$. Define the system $(X,\mu,T)$ as follows \begin{enumerate}
  \item[$\cdot$] $X:=\{0,1\}^G$ with the product topology, meaning $X$ is the space of all functions from $G$ into $\{0,1\}$, with the topology of pointwise convergence,
       \item[$\cdot$] $T_g$ is defined by $T_g x(h)= x(h+g)$,
        \item[$\cdot$] and $\mu$ is defined to satisfy
  \[
  \int f\, d\mu = m(\phi_f)
  \]
  for every continuous $f:X\to \mathbb C$, where $\phi_f:G\to \mathbb C$ is given by $\phi_f(g)=f(1_{A+g})$.
  \end{enumerate}
  Having defined $\int f\, d\mu$ for continuous $f$, the Riesz representation theorem shows that there really is a Borel measure $\mu$ satisfying this definition.  We now show that our definition of $\mu$ agrees with the definition of $\mu$ in \cite[Theorem 2.1]{BergelsonMcCutcheon}.  We do so by showing that the two definitions agree on the cylinder sets in $X$.  Writing such a cylinder set $C$ by fixing $h_1,\dots,h_k\in G$ and setting $C:=\{x\in X: x(h_1)=\varepsilon_1,\dots, x(h_k)=\varepsilon_k\}$, we get that $\phi_f(g)=1$ if and only if $g\in (A_1-h_1)\cap (A_2-h_2)\cap \dots\cap (A_k-h_k)$, where $A_i=A$ if $\varepsilon_i=1$ and $G\setminus A$ if $\varepsilon_i=0$.  Thus $\mu(C) = m(\phi_f)=m(\prod_{i=1}^k 1_{A_i-h_i})$, and the latter is the definition of $\mu(C)$ in \cite{BergelsonMcCutcheon}.  Following the argument in \cite{BergelsonMcCutcheon} proves inequality (\ref{eqn:Correspondence}).
  
  To prove ergodicity of $(X,\mu,T)$ we will prove that for all $f_1, f_2\in L^\infty(\mu)$, we have 
  \begin{equation}\label{eqn:ToProveErgodic}\lim_{n\to\infty} \frac{1}{|F_n|}\sum_{g\in F_n} \int f_1\cdot T_gf_2\, d\mu=\int f_1\, d\mu \int f_2\, d\mu
   \end{equation} for each F{\o}lner sequence $(F_n)_{n\in \mathbb N}$.  To see this, note first that it suffices to prove this identity for continuous $f_i$, by approximating bounded $f_i$ with continuous functions.  Now choose an invariant mean $\eta$ on $\ell^\infty(G)$ with the property that $\lim_{n\to\infty} \frac{1}{|F_n|}\sum_{g\in F_n}\phi(g)=\eta(\phi)$ whenever the limit on the left exists.  With this choice of $\eta$, the definition of $\mu$ lets us write the left-hand side of (\ref{eqn:ToProveErgodic}) as
   \[
   \int \phi_{f_1}(h) \phi_{f_2}(h+g)\, dm(h)\, d\eta(g)
   \]
   and the extremality of $m$ allows us to simplify this as $m(\phi_{f_1})m(\phi_{f_2})$, by (\ref{eqn:ErgodicMean}).  By the definition of $\mu$, this is just $\int f_1\,d\mu \int f_2\, d\mu$, as desired.  
   
   Now (\ref{eqn:ErgodicMean}) implies ergodicity of $(X,\mu,T)$: if $A$ is $T$-invariant, then setting $f_1=f_2=1_A$ in (\ref{eqn:ErgodicMean}), we get $\mu(A)=\mu(A)^2$, which means $\mu(A)=0$ or $1$.
\end{proof}


\section{Ergodic theoretic machinery}\label{sec:Ergodic}

\subsection{Densities on \tpdf{$G$}{G}}\label{sec:Densities}

A sequence $\mathbf F=(F_n)_{n\in \mathbb N}$ of finite subsets of $G$ is a \emph{F{\o}lner sequence} if for all $g\in G$, $\lim_{n\to \infty \frac{|F_n\triangle (g+F_n)|}{|F_n|}}=0$.  If $\mathbf F$ is a F{\o}lner sequence and $A\subset G$ the \emph{upper density of $A$ with respect to $\mb F$} is $\bar{d}_{\mb F}(A):\limsup_{n\to \infty} \frac{|A\cap F_n|}{|F_n|}$.  The \emph{upper Banach density of $A$} is $d^*(A):=\sup\{d_{\mb F}(A) : \mb F \text{ is a F{\o}lner sequence}\}$.

\subsection{Measure preserving systems}\label{sec:MeasurePreserving}

We use the usual ergodic theoretic setup for dealing with problems involving configurations in dense subsets of abelian groups; see any of \cite{FurstenbergBook}, \cite{EinsiedlerWard}, \cite{McCutcheonBook}, \cite{KerrLi} for background.

Fix a countable abelian group $G$.  We say that $(X,\mu,T)$ is a \emph{measure preserving $G$-system} (or ``$G$-system'') if $(X,\mu)$ is a probability measure space, and $G$ acts on $X$ by transformations $T_g$ which preserve $\mu$.  If $f\in L^2(\mu)$ and $g\in G$, we write $T_g$ for the element $f\circ T_g$.  Now we have an action of $G$ on $L^2(\mu)$ by unitary operators $T_g$.

We say that $f$ is \emph{$T$-invariant} if $f(T_gx) = f(x)$ for $\mu$-almost every $x$ and all $g\in G$.

A subset $B\subset X$ is \emph{$T$-invariant} if $\mu(B \triangle T_gB) = 0$ for every $g\in G$.

The system $(X,\mu,T)$ is \emph{ergodic} if the only $T$-invariant sets $B$ have $\mu(B)=0$ or $\mu(B)=1$.

We say that $(X,\mu,T)$ has \emph{finitely many ergodic components} if the closed  subspace of $L^2(\mu)$ consisting of $T$-invariant functions is finite dimensional.

\begin{lemma}\label{lem:FinIndexFinComponents}
If $(X,\mu,T)$ is ergodic and $H$ is a finite index subgroup of $G$, then the $H$-action given by restricting $T_g$ to $g\in H$ has finitely many ergodic components.
\end{lemma}

\begin{proof}
  Let $B\subset X$ be $H$-invariant, and let $g_1,\dots, g_k$ be coset representatives of $H$.  Now
  \[X\sim_{\mu} \bigcup_{g\in G} T_g B \sim_{\mu} \bigcup_{i=1}^k T_{g_i} B,\] so $\mu(B)\geq \frac{1}{k}$.  Thus $\inf\{\mu(B):B\subset X \text{ is } H\text{-invariant}\}>0$, and the algebra $\mathcal I_H$ of $H$-invariant sets is generated (up to $\mu$-measure $0$) by finitely many atoms $B_1,\dots,B_r$.  Now every $H$-invariant $f\in L^2(\mu)$ is $\mathcal I_H$-measurable, so the functions $\psi_i:=\frac{1}{\mu(B_i)}1_{B_i}$ form an orthonormal basis of the space of $H$-invariant elements of $L^2(\mu)$.
\end{proof}

We will use the Mean Ergodic Theorem for unitary actions; see \cite[Theorem 4.22]{KerrLi} for exposition.

\begin{theorem}[Mean ergodic theorem]\label{thm:Ergodic}
Let $U$ be an action of $G$ on a Hilbert space $\mathcal H$ by unitary operators and $x\in \mathcal H$.  If $(F_n)_{n\in \mathbb N}$ is a F{\o}lner sequence, then
\[
\lim_{n\to\infty} \frac{1}{F_n}\sum_{g\in F_n} U_gx = P_Ix,
\]
where $P_Ix$ is the orthogonal projection of $x$ onto the closed subspace of $\mathcal H$ consisting of $U$-invariant vectors.
\end{theorem}

Specializing to the unitary action associated to a $G$-system yields the following.

\begin{corollary}\label{cor:Ergodic}
  Let $(X,\mu,T)$ be a $G$-system and $f\in L^2(\mu)$.  Then
  \[\lim_{n\to \infty} \frac{1}{|F_n|}\sum_{g\in F_n} T_gf =P_{\mathcal I}f,\] where $P_{\mathcal I}$ is the orthogonal projection onto the closed subspace of $L^2(\mu)$ consisting of $T$-invariant functions.
\end{corollary}

\begin{observation}\label{obs:PIfinite} If $(X,\mu,T)$ has only finitely many ergodic components, then $X$ can be partitioned into $\mu$-measurable $G$-invariant subsets $X_1,\dots,X_k \subset X$, and for each $f\in L^2(\mu)$, the projection $P_{\mathcal I}f$ can be written as
\[
\sum_{i=1}^k \Bigl(\frac{1}{\mu(X_i)}\int f\, 1_{X_i} \, d\mu\Bigr) 1_{X_i}.
\]
\end{observation}

\subsection{Eigenfunctions}  We say that $f\in L^2(\mu)$ is an \emph{eigenfunction} of $(X,\mu,T)$ if there is a character $\chi\in \widehat{G}$ such that $T_gf \sim_\mu \chi(g)f$ for every $g\in G$.  Let $\mathcal{AP}$ denote the closure of the span of the eigenfunctions of $(X,\mu,T)$ in $L^2(\mu)$, and write $\mathcal{WM}$ for its orthogonal complement.  Note that $\mathcal{WM}$ is a closed, $T$-invariant subspace of $L^2(\mu)$.

Given a function $\psi:G\to \mathbb C$, we say that \emph{$\psi$ tends to $0$ in density} if for every F{\o}lner sequence $\mb F$, we have $\lim_{n\to\infty} \frac{1}{|F_n|}\sum_{g\in F_n}|\psi(g)|=0$.  Equivalently, $\psi\to 0$ in density if for all $\varepsilon>0$, the set $A_{\varepsilon}:=\{g:|\psi(g)|>\varepsilon\}$ has $d^*(A_{\varepsilon})=0$.

\begin{lemma}\label{lem:WM}  Let $(X,\mu,T)$ be a $G$-system,  $f\in \mathcal{WM}$ and $h\in L^2(\mu)$.

\begin{enumerate} \item[(i)] The correlation sequence $c(g):=\int f\cdot T_g h\, d\mu$ tends to $0$ in density.
\item[(ii)] If $T$ has finitely many ergodic components, then $c_{\mathcal I}(g):=\|P_{\mathcal I}(f\cdot T_g h)\|_{L^2(\mu)}$ tends to $0$ in density.
    \end{enumerate}
\end{lemma}

\begin{proof}
Part (i) here follows from Corollary 1.5 of \cite{BergelsonRosenblatt} (cf. Proposition 2.20 and Propsition D.17 of \cite{KerrLi}).  Part (ii) follows from Part (i) and Observation \ref{obs:PIfinite}.
\end{proof}

As usual for ergodic theoretic proofs of Roth's theorem, we need a van der Corput lemma.  The version we use is \cite[Theorem 2.12]{BergelsonMoreira}, specialized to countable discrete abelian groups; cf. \cite[Lemma 4.2]{BMZ}.

\begin{lemma}[van der Corput lemma]\label{lem:vdC}
Let $(x_g)_{g\in G}$ be a bounded collection of elements of a Hilbert space $\mathcal H$ and $\mb F$ be a F{\o}lner sequence.  If
\[
\lim_{k\to\infty }\frac{1}{|F_k|} \sum_{h\in F_k} \lim_{n\to \infty}\frac{1}{|F_n|} \sum_{g\in F_n} \langle x_{g+h}, x_g\rangle = 0
\]
then $\lim_{n\to\infty}\Bigl\|\frac{1}{|F_n|}\sum_{g\in F_n} x_g\Bigr\|_{\mathcal H}= 0$.
\end{lemma}

The next lemma will be used in conjunction with Lemma \ref{lem:vdC}. 

\begin{lemma}\label{lem:WMCorrelations}
  If $f_1\in \mathcal{WM}$, $f_0\in \mathcal{AP}$, and $bG$ has finite index in $G$, then
  \[
  \lim_{n\to \infty} \frac{1}{|F_n|} \sum_{g\in F_n}T_{ag}f_{0}\,T_{bg}f_{1}=0
  \]
  in $L^2(\mu)$.
\end{lemma}
\begin{proof}
  Since $f_0\in \mathcal{AP}$, it is a linear combination of eigenfunctions $\psi$, and for each such $\psi$ there is a character $\chi$ of $G$ so that $T_{g}\psi=\chi(g)\psi$.  We are therefore reduced to proving that when $\chi$ is a character of $g$,
  \[
  \lim_{n\to \infty} \frac{1}{|F_n|}\sum_{g\in F_n} \chi(g)T_{bg}f_{1}=0
  \]
  in $L^2(\mu)$.  Applying the mean ergodic theorem to the unitary action $U$ on $L^2(\mu)$ defined by $U_gf = \chi(g)T_{bg}f$, we get that the limit of the averages above is the orthogonal projection of $f_1$ onto the space of $U_g$-invariant functions, which is contained in $\mathcal{AP}$.  Since $f_1\perp \mathcal{AP}$, the limit is $0$.
\end{proof}

\begin{lemma}\label{lem:RothReduce}
  If $p\in L^2(\mu)$ and $q\in \mathcal{WM}$, $a,b\in G$, and both $bG$ and $(b-a)G$ have finite index in $G$, then
  \begin{equation}\label{eqn:agbg}
  \lim_{n\to\infty} \frac{1}{|F_n|}\sum_{g\in F_n} T_{ag}p\,T_{bg}q= 0
  \end{equation}
  in $L^2(\mu)$.
\end{lemma}
\begin{proof}  We will apply Lemma \ref{lem:vdC} with $x_g=T_{ag}p\, T_{bg}q$.  First we write $\langle x_{g+h}, x_g\rangle_{L^2(\mu)}$ as
  \[
  \int T_{ag+ah}p\cdot T_{bg+bh}q\cdot T_{ag}p\cdot T_{bg}q\, d\mu = \int p T_{ah}p \cdot T_{(b-a)g}(qT_{bh}q)\, d\mu
  \]
  Averaging over $g\in F_n$ and taking the limit, we get
  \begin{align*}
  \lim_{n\to \infty} \frac{1}{|F_n|}\sum_{g\in F_n} \langle x_{g+h}, x_g \rangle &= \int p \, T_{ah}p  \cdot \lim_{n\to \infty} \frac{1}{|F_n|}\sum_{g\in F_n} T_{(b-a)g}(q\,T_{bh}q)\, d\mu \\
  &= \int p\, T_{ah}p \cdot P_{\mathcal I_{b-a}}(q\,T_{bh}q)\, d\mu\\
  &\leq \|pT_{ah}p\|\, \|P_{\mathcal I_{b-a}}(q\, T_{bh}q)\|,
  \end{align*}
where $P_{\mathcal I_{b-a}}$ is the orthogonal projection onto the space of functions invariant under $g\mapsto T_{(b-a)g}$. Since  $T$ is ergodic and $(b-a)G$ has finite index in $G$, Lemma \ref{lem:FinIndexFinComponents} implies this action has finitely many ergodic components.  By Lemma \ref{lem:WM} we have that $c(h):=\|P_{\mathcal I_{b-a}}(q\, T_{bh}q)\|$ tends to $0$ in density, so
\[
\lim_{k\to\infty} \frac{1}{|F_k|}\lim_{n\to \infty} \frac{1}{|F_n|}\sum_{h\in F_k, g\in F_n} \langle x_{g+h}, x_{g} \rangle,
\]
and Lemma \ref{lem:vdC} implies equation (\ref{eqn:agbg}) holds.
\end{proof}

\subsection{Discrete spectrum}

An ergodic $G$-system $(X,\mu,T)$ has \emph{discrete spectrum} if its eigenfunctions span a dense subspace of $L^2(\mu)$.  The Halmos-von Neumann theorem\footnote{See \cite{HvN} for the case where $G=\mathbb Z$, \cite{MackeyRotation} for general groups.} states that every such system is isomorphic to an ergodic group rotation system $(K,m,R)$, meaning $K$ is a compact abelian group with Haar probability measure $m$ and $R$ is given by a homomorphism $\rho:G\to K$ with dense image, meaning $R_g(k) = k+\rho(g)$.

Multiple ergodic averages for ergodic group rotation systems can be computed as follows: if $f_i\in L^\infty(m)$, $(F_n)_{n\in \mathbb N}$ is a F{\o}lner sequence and $c_{i}\in \mathbb Z$, then
\begin{equation}\label{eqn:KroneckerBirkhoff}
  \lim_{N\to\infty} \frac{1}{|F_n|}\sum_{g\in F_n} \prod_{i=1}^df_i(k+c_i\rho(g)) = \int_K \prod_{i=1}^d f_i(k+c_is)\, ds
\end{equation}
in $L^2(m)$; cf. Section 3 of \cite{F77}.

\subsection{Factors}

The only factor we need in this article is the Kronecker factor of an ergodic $G$-system, but we mention the general theory to fix notation.  See \cite{EinsiedlerWard}, \cite{FurstenbergBook}, or \cite{KerrLi} for a general discussion of factors of dynamical systems.

A \emph{factor} of a $G$-system $(X,\mu,T)$ is a $G$-system $(Y,\nu,S)$ together with a measurable map $\pi:X\to Y$ satisfying $\mu(\pi^{-1}B)=\nu(B)$ for all $\nu$-measurable $B$ and $\pi(T_g x)=S_g(\pi(x))$ for $\mu$-almost every $x$.  

To define the Kronecker factor, we need a standard result such as \cite[Theorem 1.7]{KerrLi}, which we paraphrase as follows.

\begin{theorem}\label{thm:FactorFromAlgebra} Let $(X,\mu,T)$ be an ergodic $G$-system.  To every $T$-invariant $\sigma$-algebra $\mathcal A$ of $\mu$-measurable sets, there is a factor $(Y,\nu,S)$ with $\pi:X\to Y$ such that $A\in \mathcal A$ if and only if there is a $\nu$-measurable $B\subset Y$ with $\mu(A\triangle \pi^{-1}B)=0$.  In other words, the elements of $\mathcal A$ are, up to $\mu$-measure $0$, the $\mu$-measurable subsets of the form $\pi^{-1}B$ for $\nu$-measurable $B$.
\end{theorem}

\subsection{The Kronecker factor} Every ergodic $G$-system $(X,\mu,T)$ has a factor $(K,m,R)$ with factor map $\pi:X\to K$ satisfying the following two properties:
\begin{enumerate}
\item[(i)] $(K,m,R)$ is an ergodic group rotation $G$-system.

\item[(ii)] Every bounded $f\in \mathcal AP$ is equal $\mu$-almost everywhere to a function of the form $\tilde{f}\circ \pi$.
\end{enumerate}
With the identification in (ii), we have
\begin{equation}\label{eqn:Kronecker}
  f(T_gx) = \tilde{f}(\pi(x)+\rho(g)) \qquad \text{for } \mu\text{-a.e.} x, \text{ for all } g\in G.
\end{equation}
This factor can be obtained as the factor associated by Theorem \ref{thm:FactorFromAlgebra} to the smallest sub-$\sigma$-algebra of $(X,\mu,T)$ with respect to which every eigenfunction of $T$ is measurable.

See \cite{F77}, \cite{FurstenbergBook}, or \cite{EinsiedlerWard} for discussion of the Kronecker factor.  In \cite{BFApprox} the Kronecker factor is discussed as the ``Kronecker-Mackey factor.''

We will need two standard properties of the projection map $P_{\mathcal{AP}}:L^2(\mu)\to \mathcal{AP}$.
\begin{enumerate}
\item[] (Positivity) if $f\geq 0$ $\mu$-almost everywhere then $P_{\mathcal{AP}}\geq 0$ $\mu$-almost everywhere, and if $f:X\to [0,1]$, then $P_{\mathcal{AP}}f(x)\in [0,1]$ for $\mu$-almost every $x$.

\item[] (Integrals are preserved) $\int P_{\mathcal{AP}}f\, d\mu = \int f\, d\mu$ for all $f\in L^2(\mu)$.
\end{enumerate}
These both follow from the fact that $\mathcal AP$ is a norm-closed algebra of functions in $L^2(\mu)$ containing the constants.

\subsection{Proof of Theorem \ref{thm:DynamicalVersion}}\label{sec:ProofOfErgodicVersion}

We recall the statement of Theorem 2.2:

Let $(X,\mu,T)$ be an ergodic measure preserving $G$-system and let $(c_1,\dots,c_d)\in \mathbb Z^d$ be coefficients satisfying hypotheses (1)-(3) of Theorem \ref{thm:GeneralBogoliouboff}.  If $f_1,\dots, f_d:X\to [0,1]$ are measurable functions, $L:G^d\to [0,1]$ is defined as
\begin{equation}\label{eqn:Ldef2}
L(g_1,\dots,g_{d-1};t):= \int T_{t-\sum_{j<d} c_jg_j}f_{d}\prod_{j< d} T_{c_dg_j}f_{j} \, d\mu
\end{equation}
 and $I(t)$ is defined as the iterated limit
\begin{equation}\label{eqn:defIg2}
I(t):=\lim_{n_1\to \infty} \dots  \lim_{n_{d-1}\to \infty} \frac{1}{|F_{n_1}|\cdots |F_{n_{d-1}}|} \sum_{\substack{g_1\in F_{n_1}\\ \dots \\ g_{d-1}\in F_{n_{d-1}}}}L(g_1,\dots,g_{d-1};t)
\end{equation}
 then there are \begin{enumerate}
\item[$\bullet$] a compact metrizable abelian group $K$ with Haar measure $m$,

\item[$\bullet$] a homomorphism $\rho:G\to K$ having $\overline{\rho(G)}=K$,

\item[$\bullet$] functions $\tilde{f}_i:K\to [0,1]$ with $\int \tilde{f}_i\, dm = \int f_i\, d\mu $ for $i\leq d$
\end{enumerate}
such that
\begin{equation}\label{eqn:IgReduction2}
I(t):=\int \tilde{f}_d\Bigl(k+\rho(t)-\sum_{j<d}c_js_j\Bigr) \prod_{j<d} \tilde{f}_j(k+c_ds_j) \, dk\, ds_1\, \dots\, ds_{d-1}.
\end{equation}
Furthermore, if $f_1=\cdots = f_d$  then $\tilde{f}_1=\cdots = \tilde{f}_d$.

\begin{proof} Let $(K,m,R_\rho)$ be the Kronecker factor of $(X,\mu,T)$ with factor map $\pi:X\to K$. Let $\mathcal{AP}$ denote the set of $T$-almost periodic functions, so that every bounded $f\in \mathcal{AP}$ can be written as $\tilde{f}\circ \pi$, where $\tilde{f}\in L^\infty(m)$.

The main step of the proof is to show that the innermost limit in the definition of $I_g$ is unaffected when $f_d$ is  replaced by $P_{\mathcal{AP}}f_d$.  To prove this it suffices to prove that the innermost limit is $0$ when $P_{\mathcal{AP}}f_{d}=0$, i.e. when $f_{d}\in \mathcal{WM}$.  Assuming $f_{d}\in \mathcal{WM}$, we write the innermost average in $I_t$ as
\[
\frac{1}{|F_{n}|} \sum_{g_{d-1}\in F_n} \int T_{t-\sum_{j<d}c_jg_j}f_{d} \cdot T_{c_dg_{d-1}}f_{d-1} \prod_{j<d-1} T_{c_dg_j}f_{j}\, d\mu.
\]
We rewrite the first factor $T_{t-\sum_{j<d}c_jg_j}f_{d}$ as $T_{c_{d-1}g_{d-1}}f_d'$, where $f_d':=T_{t-\sum_{j<d-1}c_jg_j}f_{d}$, so the above expression can be written as
\begin{equation}\label{eqn:parens}
\int \Bigl(\frac{1}{|F_{n}|} \sum_{g_{d-1}\in F_n}T_{-c_{d-1}g_{d-1}}f_d'\cdot T_{c_dg_{d-1}}f_{d-1} \Bigr) \prod_{j<d-1} T_{c_dg_j}f_{j}\, d\mu.
\end{equation}
Noting that $(c_d-(-c_{d-1}))G$ has finite index in $G$ by hypothesis (\ref{hyp:cicj}) in Theorem \ref{thm:GeneralBogoliouboff}, we may apply Lemma \ref{lem:RothReduce} to the average in parentheses in (\ref{eqn:parens}) and conclude that it converges to $0$ in $L^2(\mu)$ as $n\to \infty$.

The same argument shows that $f_{d-1}$ in (\ref{eqn:defIg2}) can be replaced by $P_{\mathcal AP}f_{d-1}$ without affecting the innermost limit.

Thus we may assume $f_{j}\in \mathcal{AP}$ for $j = d$ and $d-1$.  Under this assumption, Lemma \ref{lem:WMCorrelations} allows us to replace the remaining $f_j$, $1\leq j < d-1$ with $P_{\mathcal{AP}} f_j$ without affecting the limit in (\ref{eqn:defIg2}).

Having replaced each $f_j$ in (\ref{eqn:defIg2}) with $P_{\mathcal{AP}} f_j$, we write $P_{\mathcal{AP}}f_j$ as $\tilde{f}_j\circ \pi$, and note that equation (\ref{eqn:Kronecker}) means $T_hP_{\mathcal{AP}} f_j = \tilde{f}_j(\pi(x)+\rho(h))$ for $\mu$-almost every $x$. We can compute the innermost limit in (\ref{eqn:defIg2}) by (\ref{eqn:KroneckerBirkhoff}) to get
\begin{align*}
  &\lim_{n_{d-1}\to \infty} \frac{1}{|F_{n_{d-1}}|}  \sum_{g_{d-1}\in F_{n_{d-1}}} \int \tilde{f}_d(k+\rho(g-\sum_{j<d} c_jg_j))\prod_{j<d} \tilde{f}_j(k+\rho(c_dg_j))\, dk\\
  &= \int \tilde{f}(k+\rho(g-\sum_{j<d-1} c_jg_{j-1})+c_{d-1}s_{d-1})\prod_{j<d-1} \tilde{f}_j(k+\rho(c_dg_j))\, \tilde{f}_{d-1}(k+c_ds_{d-1})\, dk\, ds_{d-1}.
\end{align*}
The same computation for averages over $g_{j}\in F_{n_j}$, $j=d-2$, $j=d-3$, $\dots$, $j=1$, will simplify the remaining limits in (\ref{eqn:defIg}) to obtain (\ref{eqn:IgReduction2}).

With our definition of the $\tilde{f}_i$ we get that $\tilde{f}_i=\tilde{f}_j$ whenever $f_i=f_j$. \end{proof}

\frenchspacing
\bibliographystyle{plain}
\bibliography{BohrNhoodGeneralBib}

\end{document}